\documentclass[letterpaper, 10 pt, conference]{ieeeconf}  

\IEEEoverridecommandlockouts                              
\usepackage{cite}
\usepackage{amsmath,amssymb,amsfonts}
\usepackage{bbm}
\usepackage{mathrsfs}
\usepackage{algorithm}
\usepackage{algpseudocode}
\usepackage{array}
\usepackage{graphicx}
\usepackage{multicol}
\usepackage{multirow}
\usepackage{epstopdf}
\usepackage{makecell}
\usepackage{textcomp}
\usepackage{changes}
\usepackage{booktabs}
\usepackage{cases}
\usepackage{soul}
\usepackage{url}
\usepackage{stfloats}
\usepackage{verbatim}
\usepackage{subfigure}
\usepackage{siunitx}

\newtheorem{theorem}{Theorem}

\title{\LARGE \bf
Data-Driven Predictive Control Using Closed-Loop Data: An Instrumental Variable Approach
}

\author{Yibo Wang, Yiwen Qiu, Malika Sader, Dexian Huang, and Chao Shang, \IEEEmembership{Member, IEEE}
\thanks{This work was supported by National Natural Science Foundation of China under Grant 62373211. \textit{(Corresponding author: Chao Shang.)}}
\thanks{Y. Wang and M. Sader are with Department of Automation, Tsinghua University, Beijing 100084, China (e-mail: wyb21@mails.tsinghua.edu.cn, mlksdr@tsinghua.edu.cn). }
\thanks{Y. Qiu is with Carnegie Mellon University,
5000 Forbes Avenue, Pittsburgh, PA 15213, USA
(e-mail: yiwenq@andrew.cmu.edu)}
\thanks{D. Huang and C. Shang are with Department of Automation, Beijing National Research Center for Information Science and Technology, Tsinghua University, Beijing 100084, China (e-mail: huangdx@tsinghua.edu.cn, c-shang@tsinghua.edu.cn). }}

\begin{document}

\maketitle
\thispagestyle{empty}
\pagestyle{empty}

\begin{abstract}
   Current data-driven predictive control (DDPC) methods heavily rely on data collected in open-loop operation with elaborate design of inputs. However, due to safety or economic concerns, systems may have to be under feedback control, where only closed-loop data are available. In this context, it remains challenging to implement DDPC using closed-loop data. In this paper, we propose a new DDPC method using closed-loop data by means of instrumental variables (IVs). By drawing from closed-loop subspace identification, the use of two forms of IVs is suggested to address the closed-loop issues caused by feedback control and the correlation between inputs and noise. Furthermore, a new DDPC formulation with a novel IV-inspired regularizer is proposed, where a balance between control cost minimization and weighted least-squares data fitting can be made for improvement of control performance. Numerical examples and application to a simulated industrial furnace showcase the improved performance of the proposed DDPC based on closed-loop data.
\end{abstract}

\section{Introduction}
\label{sec:introduction}
In recent years, data-driven methods have received widespread attentions across various fields in systems and control \cite{hou2013model,shang2019data}. From the control perspective, a large number of recent results on data-driven predictive control (DDPC) methods have emerged built upon the behavioral systems theory \cite{willems2005note,markovsky2021behavioral}. A salient feature of DDPC methods is that control policies can be directly attained from raw data without the need for model identification \cite{markovsky2016missing}, while preserving the capability of constraint handling and the robustness of receding horizon implementation of the classical model-based predictive control scheme \cite{breschi2023data}.


For most DDPC methods, in order to ensure the requirement of persistent exciting data, data collection is typically carried out in open-loop operation with elaborate design of inputs \cite{markovsky2021behavioral,markovsky2016missing}. Yet, open-loop data collection is not always possible in practice since the process may be unstable or have to be operated under feedback control due to safety or economic concerns. In such cases, to handle constraints and attain superior performance, an advanced data-driven predictive controller may be favored in place of current controller that has been ``coarsely" tuned to ensure closed-loop stability. However, it may be problematic to implementing DDPC with closed-loop data. As documented in \cite{van2022data}, the basic form of DDPC can be interpreted as using a particular choice of instrumental variable (IV), which shall be uncorrelated with future innovations in open-loop conditions. However, this no longer holds in closed-loop conditions due to the correlation between future inputs and innovations, which eventually yields biased output predictions \cite{dinkla2023closed,huang2008dynamic}. Thus, how to implement DDPC using closed-loop data remains a significant challenge.

In this work, a new DDPC method using closed-loop data is proposed based on the DDPC framework in \cite{van2022data} by means of IVs. By borrowing from closed-loop subspace identification (SID), we propose the use of two specific forms of IV in DDPC based on closed-loop data. The first one is the future feedback reference, which has been generically used in SID to eliminate the bias induced by feedback \cite{huang2005closed}. The second one is a more sophisticated one designed using the left coprime factorization (LCF) of controller \cite{li2020novel}, which has been adopted in closed-loop SID with significant performance improvement. A combined use of both IVs helps eliminating the effect of noise while retaining useful information within future inputs and outputs. Furthermore, for the IV based on feedback reference, we point out that it helps to disentangle the blending of controller dynamics and plant dynamics with input/output data, which is a critical challenge for DDPC with closed-loop data but still remains uncovered in literature. Based on this, we further propose a new regularized DDPC formulation with a novel IV-inspired regularizer, where a balance between minimizing control cost and fitting data via a weighted least-squares criterion can be made to improve tracking performance. Numerical examples and application to a simulated tubular furnace system demonstrate that, the proposed method offers possibility of implementing DDPC merely using closed-loop data.


The rest of this work is structured as follows. Section II gives a brief introduction of the DDPC with IV scheme. In Section III, the choices of IV based on future reference and controller information are discussed, followed by a regularized DDPC with IV method. Results of case studies are reported in Sections IV and V, followed by final conclusions.

\textbf{Notation:} Given a sequence $\{x(i)\}_{i=1}^N$, $x_{[i,j]}$ denotes its restriction of to the interval $[i,j]$, i.e. $x_{[i,j]} = {\rm{col}}(x(i),...,x(j))$. The Hankel matrix operator $\mathcal{H}_s(x_{[i,j]})$ is used to construct the block Hankel matrix of depth $s$. For a state-space model $(A,B,C,D)$, $\Gamma_{s}(A,C)$ and $\Delta_s(A,B)$ denote the extended observability and controllability matrices of order $s$. A lower block-triangular Toeplitz matrix of depth $s$ can be constructed using the operator $\mathcal{T}_{s}(A,B,C,D)$. $\|\cdot\|_F$ denotes the Frobenius norm. Given an integer $s$, $1_s\in\mathbb{R}^s$ denotes the vector of all ones, and $I_s$ denotes the identity matrix of size $s$. The Kronecker product of two matrices is indicated by $X \otimes Y$.

\section{Preliminaries}

Consider the discrete-time linear time-invariant (LTI) system
\begin{equation}
    \label{equation: LTI system in free noise}
    \begin{aligned}
        x(t+1)&=Ax(t)+Bu(t),\\
        y(t)&=Cx(t)+Du(t),
    \end{aligned}
\end{equation}
where $x(t)\in\mathbb{R}^n$, $u(t)\in\mathbb{R}^m$ and $y(t)\in\mathbb{R}^p$ denote state, input and output, respectively. It is assumed that the system \eqref{equation: LTI system in free noise} is minimize. Given an input-output trajectory $\{u^d(i),y^d(i)\}_{i=1}^N$ from \eqref{equation: LTI system in free noise}, the well-known DDPC can be formulated as a constrained optimization problem at time $t$ with past horizon $L_p \ge n$ and future horizon $L_f \ge 1$ \cite{coulson2019data}:
\begin{subequations}
    \label{equation: DeePC}
    \begin{align}
        \min_{u_f,\hat{y}_f,g}\ &\mathcal{J}(u_f,\hat{y}_f)\\
        \rm{s.t.}\ \ &
        \begin{bmatrix}
            Z_p\\U_f\\Y_f   
        \end{bmatrix} g=
        \begin{bmatrix}
            z_p\\u_f\\\hat{y}_f
        \end{bmatrix}, \label{eq: 2b} \\
        & u_f \in \mathbb{U},~ \hat{y}_f \in \mathbb{Y}
    \end{align}
\end{subequations}
where
\begin{equation*}
    \begin{aligned}
        &Z_p={\rm col}(U_p,Y_p),~z_p(t)={\rm col}(u_p,y_p),\\
        &U_p=\mathcal{H}_{L_p}(u^d_{[1,N-L_f]}),~U_f=\mathcal{H}_{L_f}(u^d_{[L_p+1,N]}),\\
        &u_p={\rm col}(u_{[t-L_p,t-1]}),~u_f={\rm col}(u_{[t,t+L_f-1]}),
    \end{aligned}
\end{equation*}
and similarly for $Y_p$, $Y_f$, $y_p$ and $\hat{y}_f$, $g\in\mathbb{R}^{\bar{N}}$ with $\bar{N}=N-L_f-L_p+1$, and $\mathbb{U}$ and $\mathbb{Y}$ are input and output constraint sets. For future output $\hat{y}_f$, we use the symbol $\hat{\cdot}$ to stress its extrapolating nature. The objective of \eqref{equation: DeePC} can be defined as the standard quadratic cost:
\begin{equation}
    \label{equation: cost function}
    \begin{aligned}
    \mathcal{J}(u_f,\hat{y}_f)=\|\hat{y}_f-y^r_f\|_Q^2+\|u_f\|_R^2,
    \end{aligned}
\end{equation}
where $Q,R\succ0$ are weighting matrices and $y^r_f$ denotes the future output reference. 

In the presence of process disturbance and measurement noise, we consider the stochastic system expressed in an innovation form:
\begin{equation}
    \label{equation: system in innovation form}
    \begin{aligned}
        x(t+1)&=Ax(t)+Bu(t)+Ke(t),\\
        y(t)&=Cx(t)+Du(t)+e(t),
    \end{aligned}
\end{equation}
where $K$ is the steady state Kalman gain, and the innovation $e(t)$ is zero-mean white noise sequence. The matrix $A_K\triangleq A-KC$ is assumed to be strictly stable and $\|A_K^{L_p}\|_F\approx0$ holds for a sufficiently large $L_p$ \cite{chiuso2007role}. In this case, the following input/output data equation has been widely used in SID and data-driven control \cite{breschi2023data}:
\begin{equation}
    \label{equation: data equation}
    \begin{aligned}        &Y_f=\Gamma_{L_f}\underbrace{\begin{bmatrix}\Delta_{L_p}^u&\Delta_{L_p}^y\end{bmatrix}}_{\triangleq \Delta_{L_p}} Z_p+H_{L_f}^uU_f+H_{L_f}^eE_f,
    \end{aligned}
\end{equation}
where $\Gamma_{L_f}$ is the observability matrix, $\Delta_{L_p}^u=\Delta_{L_p}(A,B)$ and $\Delta_{L_p}^y=\Delta_{L_p}(A,K)$ are controllability matrices, $H_{L_f}^u$ and $H_{L_f}^e$ are block Toeplitz matrices constructed of $(A,B,C,D)$ and $(A,K,C,I)$. The Hankel matrix $E_f=\mathcal{H}_{L_f}(e^d_{[L_p+1:N]})$ encodes the uncertainty arising from future innovations. As for online data sequence from \eqref{equation: system in innovation form}, a similar relation stands:
\begin{equation}
    \label{equation: data sequence}
    \begin{aligned}
        &y_f=\Gamma_{L_f}\Delta_{L_p}z_p+H_{L_f}^uu_f+H_{L_f}^ee_f,
    \end{aligned}
\end{equation}
where $e_f={\rm col}(e_{[t,t+L_f-1]})$ denotes the future innovation. Using the behavioral relation \eqref{eq: 2b} for prediction, we obtain:
\begin{equation}
    \label{equation: output predictions}
    \begin{aligned}
        \hat{y}_f & =Y_fg\\        & = \Gamma_{L_f}\Delta_{L_p}z_p+H_{L_f}^uu_f+H_{L_f}^eE_fg\\
        & = y_f+H_{L_f}^e E_fg - H_{L_f}^e e_f.
    \end{aligned}
\end{equation}
It is clear that the error in $\hat{y}_f$ stems from two aspects, i.e. the multiplicative uncertainty $H_{L_f}^e E_fg$ depending on the solution $g$, and the future innovation $H_{L_f}^e e_f$ that is essentially inevitable.
To dispel the effect of $H_{L_f}^e E_fg$, the usage of instruments has been suggested by \cite{van2022data,dinkla2023closed}, in a similar spirit to its usage in SID. More precisely, an IV matrix $\Phi$ uncorrelated with future innovations is considered, which possesses the following asymptotic property:
\begin{equation}
    \label{equation: IV independent of noise}
    \lim_{\bar{N}\to\infty}E_f\Phi^\top \approx 0.
\end{equation}
Letting $g=\Phi^\top h$, which enforces $g$ to lie in the row space of $\Phi$, then the effect of $H_{L_f}^e E_fg$ vanishes desirably. In this way, one attains a \textit{tightened} optimal control problem \cite{van2022data}:
\begin{subequations}
    \label{equation: DeePC with IV}
    \begin{align}
        \min_{u_f,\hat{y}_f,h}\ &\mathcal{J}(u_f,\hat{y}_f)\\
        \rm{s.t.}\ \ &
        \begin{bmatrix}
            Z_p\\U_f\\Y_f   
        \end{bmatrix}\underbrace{\Phi^\top h}_{=g}=
        \begin{bmatrix}
            z_p\\u_f\\\hat{y}_f
        \end{bmatrix}, \label{eq: 9b} \\
        &~ u_f\in \mathbb{U},~ \hat{y}_f\in \mathbb{Y},
    \end{align}
\end{subequations}
where $h$ appears as a new decision variable in place of $g$. When data $\{ Z_p, U_f, Y_f\}$ are collected under open-loop conditions, a proper candidate of $\Phi$ is given by \cite{van2022data}:
\begin{equation}
    \label{equation: IV for open-loop condition}
    \Phi=\frac1{\bar{N}}{\rm col}(Z_p,U_f),
\end{equation}
where $1/\bar{N}$ is used to ensure the well-posedness of \eqref{eq: 9b} with $\bar{N}$ gradually increasing, which naturally ensures \eqref{equation: IV independent of noise}. Meanwhile, including $Z_p$ and $U_f$ in $\Phi$ ensures a high correlation with ${\rm col}(Z_p,U_f,Y_f)$, which helps to avoid the ill-posedness of \eqref{eq: 9b}. According to \cite{van2022data}, substituting \eqref{equation: IV for open-loop condition} into \eqref{equation: DeePC with IV} results in the traditional subspace predictive control (SPC) \cite{favoreel1999spc} scheme, a variant of \eqref{equation: DeePC} with $g$ being the least-norm solution. This sheds, from a new perspective, some light on the capability of SPC in handling noise within open-loop data. 

Under closed-loop control, there is always a backward impact on $U_f$ from $Y_f$ due to feedback. In this case, \eqref{equation: IV for open-loop condition} is no longer applicable due to the potential correlation between $U_f$ and $E_f$, thereby posing a critical challenge to implementing DDPC with closed-loop data \cite{dinkla2023closed}. Next, we explore some suitable specifications of IV that can hedge against the effect of $E_f$ while containing information in $U_f$ and $Y_f$.

\section{Instrumental Variable-Aided Data-Driven Control}

We assume that the control loop has already been closed by a controller $C(z)$ in the backward path, which enables to stabilize the system but may not assure a desirable tracking/disturbance rejection performance. More precisely, $C(z)$ is formulated as the following LTI system:
\begin{equation}
    \label{equation: closed-loop controller}
    \begin{aligned}
        x_c(t+1)&=A_cx_c(t)+B_c[y(t)-r(t)],\\
        u(t)&=C_cx_c(t)+D_c[y(t)-r(t)],
    \end{aligned}
\end{equation}
where $x_c(t)\in\mathbb{R}^{n_c}$ denotes internal state of the controller, and $r(t)\in\mathbb{R}^{p}$ denotes the reference signal. Two assumptions are made as follows.
\begin{enumerate}
    \item The quadruple $(A_c,B_c,C_c,D_c)$ is known.
    \item The signal $r(t)$ is independent of $e(t)$.
\end{enumerate}
Based on these assumptions under closed-loop conditions, we discuss two options of IV as substitutions of $U_f$, which are built upon future reference and controller information, respectively.

\subsection{IV based on future reference}

Indeed, Assumption 2 is rather standard in closed-loop SID, where a common choice of IV is the future reference $R_f$ of the controller \cite{pouliquen2010indirect,li2021subspace}. This naturally inspires the inclusion of $R_f$ into $\Phi$ as a substitution of $U_f$, which enables to better alleviate the effect of $E_f$. Meanwhile, there exists a certain degree of correlation between $R_f$ and $\{ U_f, Y_f \}$ under closed-loop control, which helps to eliminate the ill-posedness. 

Aside from these rationales, we further point out that the usage of $R_f$ helps to disentangle the 
coexistence of process dynamics and controller dynamics in input/output data and cancel the latter, which conceptually bears resemblance to closed-loop subspace identification \cite{huang2005closed}. Under closed-loop conditions, there exist dynamics in both forward path and backward path \cite{jiang2015simultaneous}. In the behavioral framework, however, the roles of input and output are conceptually equal \cite{markovsky2021behavioral}. Consequently, a combination of input/output data in \eqref{eq: 2b} yields an implicit characterization of \textit{bi-directional dynamics}, which may be problematic in DDPC. 



In \eqref{eq: 2b}, the equality $Z_p g = z_p$ is responsible for deciding the implicit initial condition, while constraints $U_f g = u_f$ and $Y_f g = \hat{y}_f$ describe the multi-step forward prediction. Because the main design freedom of \eqref{equation: DeePC} lies in $u_f$ and $\hat{y}_f$, our major focus is placed on analyzing the bi-directional dynamics encoded in $\{ U_f, Y_f \}$ and their influence on the control design. Akin to \eqref{equation: data equation}, the subspace relation between $U_f$ and $Y_f$ due to the controller $C(z)$ can be expressed as:
\begin{equation}
    \label{equation: subspace equation controller}
    U_f=\Gamma_{L_f}^cX_f^c+H_{L_f}^c(R_f-Y_f),
\end{equation}
where $\Gamma_{L_f}^c=\Gamma_{L_f}(A_c,C_c)$, $H_{L_f}^c=\mathcal{T}_{L_f}(A_c,B_c,C_c,D_c)$, $X_f^c$ denotes the state matrix of $C(z)$ and $R_f$ is a Hankel matrix constructed with $r(t)$. We then pre-multiple \eqref{equation: subspace equation controller} by $\Gamma_{L_f}^{c,\bot}$, which is the orthogonal column space of $\Gamma_{L_f}^c$:
\begin{equation}
    \label{equation: subspace equation controller eliminate state}
    \Gamma_{L_f}^{c,\bot}U_f+\Gamma_{L_f}^{c,\bot}H_{L_f}^cY_f=\Gamma_{L_f}^{c,\bot}H_{L_f}^cR_f.
\end{equation}
Further using IVs, \eqref{equation: subspace equation controller eliminate state} becomes:
\begin{equation}
    \label{equation: subspace equation controller eliminate state IV}
        \lim_{\bar{N}\to\infty}\Theta_c
    \begin{bmatrix}
        U_f\\Y_f    \end{bmatrix}\Phi^\top=\lim_{\bar{N}\to\infty}\Gamma_{L_f}^{c,\bot}H_{L_f}^cR_f\Phi^\top,
\end{equation}
where $\Theta_c \triangleq \Gamma_{L_f}^{c,\bot}[I~~~H_{L_f}^c]$ encodes essential information of the controller $C(z)$. Assume that the rows of $R_f$ are independent with those of $\Phi$, i.e.,
\begin{equation}
    \label{equation: independence of Rf and Z}
    \lim_{\bar{N}\to\infty}R_f\Phi^\top=0,
\end{equation}
which makes the left-hand side of \eqref{equation: subspace equation controller eliminate state IV} tend to zero. It then follows from \eqref{eq: 9b} that the multi-step forward prediction will be subject to the controller relation asymptotically:
\begin{equation}
    \label{equation: data sequence Rf Z}
    \lim_{\bar{N}\to\infty}\Theta_c
    \begin{bmatrix}
        u_f\\\hat{y}_f
    \end{bmatrix}=0.
\end{equation}
This indicates that under the condition \eqref{equation: independence of Rf and Z},  apart from the plant dynamics in the forward path, the multi-step forward prediction based upon closed-loop data has to follow an implicit relation of $C(z)$ in the backward path asymptotically, thereby leading to a \textit{restricted} behavioral representation. As a result, when solving for the optimal control design in \eqref{equation: DeePC with IV}, only a subset of input/output behaviors is taken into account, which may lead to compromised control performance. To counteract this, the choice of IV shall satisfy
\begin{equation}
    \label{equation: correlation Rf Z}
    \lim_{\bar{N}\to\infty} R_f\Phi^\top\neq0,
\end{equation}
implying that $\Phi$ ought to be as highly correlated with $R_f$ as possible. This suggests including $R_f$ into $\Phi$ for DDPC, which echoes with its wide usage in IV-based closed-loop SID.

\subsection{IV based on LCF of controller}

More recently, a novel design of IV for closed-loop SID 
was put forward by \cite{li2020novel}, which utilizes LCF of the controller to eliminate the estimation bias caused by correlation between inputs and noise. This motivates a new choice of IV in DDPC, which eliminates the effect of $E_f$ in DDPC while preserving most information in input and output data. Suppose $C(z)$ admits the following LCF:
\begin{equation}
    \label{equation: left coprime factorization}
    C(z)=V_c^{-1}(z)U_c(z),
\end{equation}
where $V_c(z)=(A_v,B_v,C_v,D_v)$ and $U_c(z)=(A_u,B_u,C_u,D_u)$. The relation between inputs and outputs from the aspect of coprime factors in \eqref{equation: left coprime factorization} can be written as:
\begin{equation}
    \label{equation: input output left coprime factorization}
    V_c(z)u(t)=U_c(z)[r_c(t)-y(t)].
\end{equation}
This yields a subspace matrix equation:
\begin{equation}
    \label{equation: subspace equation LCF}
    \begin{aligned}        \Gamma_{L_f}^vX_f^v+H_{L_f}^{c,v}U_f&=\Gamma_{L_f}^uX_f^u+H_{L_f}^{c,u}(R_f-Y_f)
    \end{aligned}
\end{equation}
where $\Gamma_{L_f}^v=\Gamma_{L_f}(A_v,C_v),~H_{L_f}^{c,v}=\mathcal{T}_{L_f}(A_v,B_v,C_v,D_v)$. $\Gamma_{L_f}^u$ and $H_{L_f}^{c,u}$ have similar definitions based on $U_c(z)$, and $X_f^v$ and $X_f^u$ denote extended state matrices of systems $V_c(z)$ and $U_c(z)$. Thus, based on \eqref{equation: subspace equation LCF}, we define the following IV:
\begin{equation}
    \label{equation: Xif}
    \begin{split}
    \Xi_f&\triangleq H_{L_f}^{c,v}U_f+H_{L_f}^{c,u}Y_f \\
        &=\Gamma_{L_f}^uX_f^u-\Gamma_{L_f}^vX_f^v+H_{L_f}^{c,u}R_f.
    \end{split}
\end{equation}
Clearly, $\Xi_f$ is expressible as a combination of future reference $R_f$ and initial states $\{ X_f^v, X_f^u \}$, all of which are uncorrelated with $E_f$ \cite{li2020novel}. Thus, $\Xi_f$ fulfills the condition \eqref{equation: IV independent of noise}. What's more, $\Xi_f$ encodes information in $U_f$ and $Y_f$ and thus can alleviate the ill-posedness of \eqref{eq: 9b}.

\subsection{IV-aided DDPC and regularization}

Based on above arguments, we use a combination of $R_f$ and $\Xi_f$, together with $Z_p$, to construct a new IV matrix for DDPC with closed-loop data:
\begin{equation}
    \label{equation: IV closed loop}
    \Phi=\frac{1}{\bar{N}}{\rm col}(Z_p,\Xi_f,R_f),
\end{equation}
where $Z_p$ is involved to preserve essential information in past input/output data, and both $R_f$ and $\Xi_f$ can help to cancel the correlation with $E_f$ while being correlated to $U_f$ and $Y_f$. By inserting \eqref{equation: IV closed loop} into \eqref{equation: DeePC with IV}, the DDPC with IV based on closed-loop data can be derived. 

Inspired by the projection-based regularization in \cite{dorfler2022bridging}, we further propose a new DDPC formulation with a novel IV-inspired regularizer with $g$ being the decision variable:
\begin{equation}
    \label{equation: regularized DeePC}
    \begin{aligned}
        \min_{u_f,\hat{y}_f,g}\ &\mathcal{J}(u_f,\hat{y}_f)+\lambda \cdot \|(I-\Pi)g\|_p\\
        \rm{s.t.}\quad&
        \begin{bmatrix}
            Z_p\\U_f\\Y_f 
        \end{bmatrix}g=
        \begin{bmatrix}
            z_p\\u_f\\\hat{y}_f
        \end{bmatrix},~ u_f \in \mathbb{U},~ \hat{y}_f \in \mathbb{Y},
    \end{aligned}
\end{equation}
where $\Pi=\Phi^\top\left(\begin{bmatrix}Z_p\\U_f\end{bmatrix}\Phi^\top\right)^\dagger\begin{bmatrix}Z_p\\U_f\end{bmatrix}$, and $\lambda\ge0$ is the regularization parameter. The rationale of this particular choice of $\Pi$ is made clear below.

\begin{theorem}
    \label{theorem: relaxation}
    Assume that $\mathbb{U}$ and $\mathbb{Y}$ are convex sets. For $\lambda\ge0$, the regularized DDPC problem \eqref{equation: regularized DeePC} is a convex relaxation of the following variant of SPC, which is an indirect DDPC formulation:
    \begin{equation}
        \label{equation: weighted SPC}
        \begin{aligned}
            \min_{u_f,\hat{y}_f}\ &\mathcal{J}(u_f,\hat{y}_f)\\
            {\rm s.t.}\ \ &\hat{y}_f=\Omega^*\begin{bmatrix}
                z_p\\u_f
            \end{bmatrix},~u_f\in\mathbb{U},~\hat{y}_f\in\mathbb{Y},\\
            &{\rm where}\ \Omega^*\ {\rm solves}\ \min_{\Omega}\ \left\|\left(Y_f-\Omega\begin{bmatrix}Z_p\\U_f\end{bmatrix}\right)\Phi^\top\right\|^2_F,
        \end{aligned}
    \end{equation}
    where $\Omega^*$ is a multi-step predictor fitted in a weighted least-squares sense based on the weighting matrix $\Phi^\top\Phi$.
\end{theorem}
\begin{proof}
    The optimum of the inner problem in \eqref{equation: weighted SPC} is:
    \begin{equation}
        \label{equation: optimal solution of predictor}
        \Phi^*=Y_f\Phi^\top\left(\begin{bmatrix}Z_p\\U_f\end{bmatrix}\Phi^\top\right)^\dagger.
    \end{equation}
    Based on \eqref{equation: optimal solution of predictor}, the inner problem of \eqref{equation: weighted SPC} amounts to:
    \begin{equation}
        \label{equation: inner problem}
        \begin{aligned}
            \hat{y}_f=Y_f\Phi^\top h^*\ {\rm where}\ h^*=\arg\min_{h}\ &\|h\|_2^2\\
            {\rm s.t.}\ &\begin{bmatrix}
                Z_p\\U_f
            \end{bmatrix}\Phi^\top h=
            \begin{bmatrix}
                z_p\\u_f
            \end{bmatrix}.
        \end{aligned}
    \end{equation}
    which can be equivalently expressed by including an orthogonality constraint:
    \begin{equation}
        \label{equation: orthogonal constraint}
        \begin{aligned}
            \begin{bmatrix}
                Z_p\\U_f\\Y_f
            \end{bmatrix}\Phi^\top h=\begin{bmatrix}
                z_p\\u_f\\\hat{y}_f
            \end{bmatrix},~ \left(I-\Tilde{\Pi}\right)h=0,
        \end{aligned}
    \end{equation}
    where $\Tilde{\Pi}=\left(\begin{bmatrix}Z_p\\U_f\end{bmatrix}\Phi^\top\right)^\dagger\begin{bmatrix}Z_p\\U_f\end{bmatrix}\Phi^\top$. Letting $g=\Phi^\top h$ for \eqref{equation: orthogonal constraint}, it follows that $\Phi^{\top} (I - \Tilde{\Pi})h = (I - \Pi)g = 0$. Then we arrive at the following relaxation of \eqref{equation: weighted SPC} as a single-level convex program:
     \begin{equation}
        \label{equation: regularized DeePC with lambda to infty}
        \begin{aligned}
            \min_{u_f,\hat{y}_f,g}\ &\mathcal{J}(u_f,\hat{y}_f)\\
        \rm{s.t.}\quad&
        \begin{bmatrix}
            Z_p\\U_f\\Y_f 
        \end{bmatrix}g=
        \begin{bmatrix}
            z_p\\u_f\\ \hat{y}_f
        \end{bmatrix},~ \|(I-\Pi)g\|_p=0,\\
        &~u_f\in\mathbb{U},~\hat{y}_f\in\mathbb{Y}.
        \end{aligned}
    \end{equation}
    By lifting the constraint $\|(I-\Pi)g\|_p=0$ into the objective via regularization, the problem \eqref{equation: regularized DeePC with lambda to infty} becomes further relaxed, thereby yielding \eqref{equation: regularized DeePC} as a convex relaxation of \eqref{equation: weighted SPC}.
\end{proof}

Indeed, Theorem \ref{theorem: relaxation} offers an extension of \cite[Theorem IV.6]{dorfler2022bridging} that bridges indirect and direct formulations of DDPC. When $\lambda\to\infty$, similar to \cite[Theorem 4]{breschi2023data}, \eqref{equation: regularized DeePC} reduces to
\begin{equation}
\label{equation: DeePC with IV inverse}
    \begin{aligned}
        \min_{u_f,\hat{y}_f}\ &\mathcal{J}(u_f,\hat{y}_f)\\
        \rm{s.t.}~~&
        \hat{y}_f=Y_f\Phi^\top\left(
        \begin{bmatrix}
            Z_p\\U_f 
        \end{bmatrix}\Phi^\top\right)^\dagger
        \begin{bmatrix}
            z_p\\u_f
        \end{bmatrix},~ u_f \in \mathbb{U},~ \hat{y}_f \in \mathbb{Y},
    \end{aligned}
\end{equation}
where the multi-step predictor is identified by weighted least-squares fitting of data, and the linking variable $g = \Phi^\top\left(\begin{bmatrix}Z_p\\U_f\end{bmatrix}\Phi^\top\right)^\dagger\begin{bmatrix}z_p\\u_f\end{bmatrix}$ is strictly limited to the row space of $\Phi$. In the case of finite $\lambda > 0$, a balance between control cost minimization and weighted least-square data fitting can be made, where $g$ in \eqref{equation: regularized DeePC} is allowed to lie outside the row space of $\Phi$, making it possible to attain improved control performance.


\section{Numerical Examples}
\label{Sec: numerical examples}

Consider the open-loop process \eqref{equation: system in innovation form} with system matrices
\begin{equation}
    \label{equation: system matrices}
    \begin{aligned}
        &A=\begin{bmatrix}
            0.7326&-0.0861\\0.1722&0.9909
        \end{bmatrix},
        ~B=\begin{bmatrix}
            0.0609\\0.0064
        \end{bmatrix},\\
        &C=\begin{bmatrix}
            0&1.4142
        \end{bmatrix},~D=1,
    \end{aligned}
\end{equation}
with that $e(t)\sim\mathcal{N}(0,\sigma_e^2)$ and $K$ is chosen such that $A_K$ is strictly stable. The plant is embedded into a standard feedback structure with the matrices of $C(z)$ as:
\begin{equation}
    \label{equation: controller matrices}
    \begin{aligned}
        &A_c=\begin{bmatrix}
            1&-0\\0.0722&1
        \end{bmatrix},
        ~B_c=\begin{bmatrix}
            0.2609\\0.164
        \end{bmatrix},\\
        &C_c=\begin{bmatrix}
            0.8&0.2142
        \end{bmatrix},
        ~D_c=-0.07.
    \end{aligned}
\end{equation}

For offline data collection, the concatenation of a series of square waves with a period of $600$, duty ratio of $70\%$ and amplitude from $-3$ to $3$ at $1$ intervals is used as the feedback reference with $N=4200$. While for DDPC, the objective is to track a square wave with a period of $60$ and amplitude of $1$. Both $L_p$ and $L_f$ are chosen as $30$. The cost weighting matrices are set as $Q = I_{L_f}$ and $R = 0.01I_{L_f}$. The input and output constraint sets are set $\mathbb{U}\in\mathbb{R}^{mL_f}$ and $\mathbb{Y}\in\mathbb{R}^{pL_f}$. To evaluate the control performance, the index $\mathcal{J}=\sum_{t=1}^{N_c} \|y(t)-r(t)\|_Q^2+\|u(t)\|_R^2$ is used,
where $N_c$ is the length of reference signal for predictive control with $N_c=60$ in the simulation. All quadratic programs are solved using the OSQP package \cite{osqp}. For a comprehensive comparison, the following control strategies are implemented.
\begin{itemize}
    \item \textbf{Oracle}: Model-based predictive control scheme with known $(A,B,C,D,K)$, where a steady-state Kalman filter is implemented for state estimation.
    \item \textbf{DDPC-IV}: The proposed DDPC \eqref{equation: DeePC with IV inverse} with IV in \eqref{equation: IV closed loop}.
    \item \textbf{RDDPC-IV}: The proposed regularized DDPC \eqref{equation: regularized DeePC}, where $\|(I-\Pi)g\|_2^2$ is used and $\lambda$ is selected within $[10^{-3},10^5]$.
    \item \textbf{DDPC-IV1}: DDPC \eqref{equation: DeePC with IV inverse} with $\Phi = {\rm col}(Z_p,R_f)$.
    \item \textbf{DDPC-IV2}: DDPC \eqref{equation: DeePC with IV inverse} with $\Phi={\rm col}(Z_p,\Xi_f)$.
    \item \textbf{SPC} \cite{favoreel1999spc,van2022data}: DDPC \eqref{equation: DeePC with IV inverse} with generic IV matrix \eqref{equation: IV for open-loop condition}.
\end{itemize}

\begin{figure}
    \centering
    \subfigure[SNR = $20$dB]{
        \includegraphics[width=0.39\textwidth]{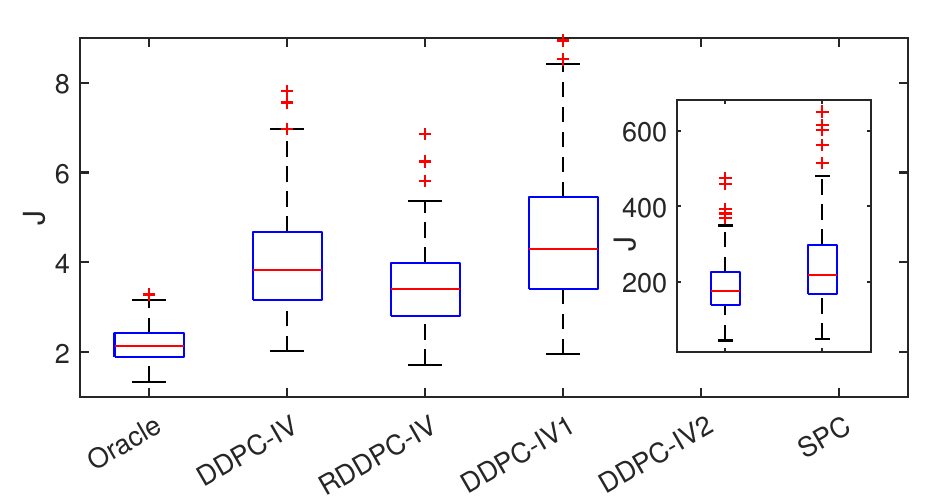}
    }
    \subfigure[SNR = $25$dB]{
        \includegraphics[width=0.39\textwidth]{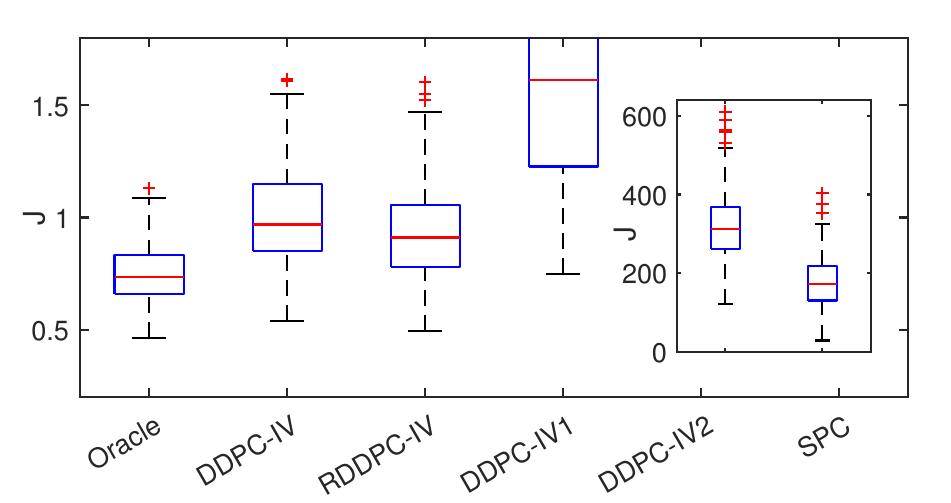}
    }
    \subfigure[SNR = $30$dB]{
        \includegraphics[width=0.39\textwidth]{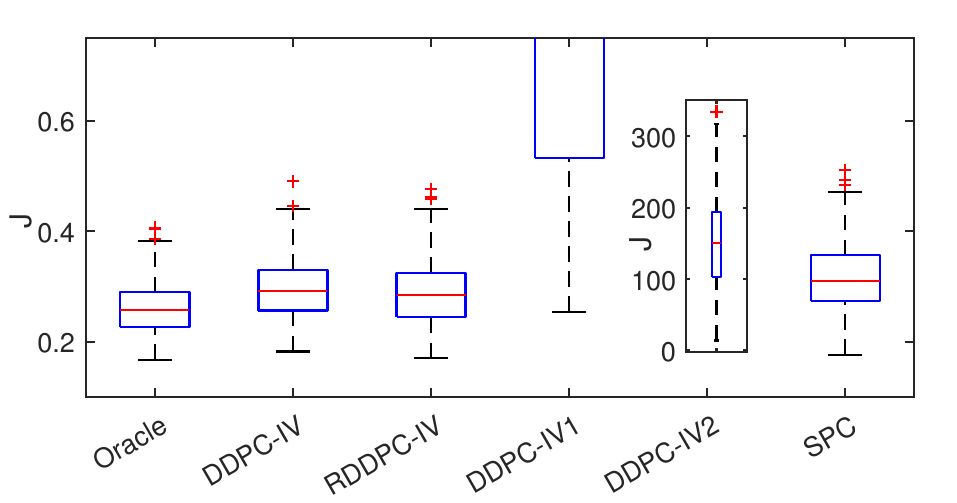}
    }
    \caption{Overall control performance of different approaches at various noise levels in $200$ Monte Carlo simulations}
    \label{fig:prediction performance}
\end{figure}

By varying $\sigma_e$, three cases corresponding to low, medium and high noise levels are created. $200$ Monte Carlo runs are carried out for each case to obtain a comprehensive evaluation of control performance. The simulation results of different algorithms are shown in Fig. \ref{fig:prediction performance}. Taken together in all noise levels, the DDPC-IV perfoms evidently better than SPC, DDPC-IV1 and DDPC-IV2, which suggests the necessity of simultaneous existence of $R_f$ and $\Xi_f$ in IV using closed-loop data. It can be observed that the RDDPC-IV performs better than DDPC-IV, which mainly owes to involvement of the IV-inspired regularizer.

\begin{figure}
    \centering
    \includegraphics[width=0.3\textwidth]{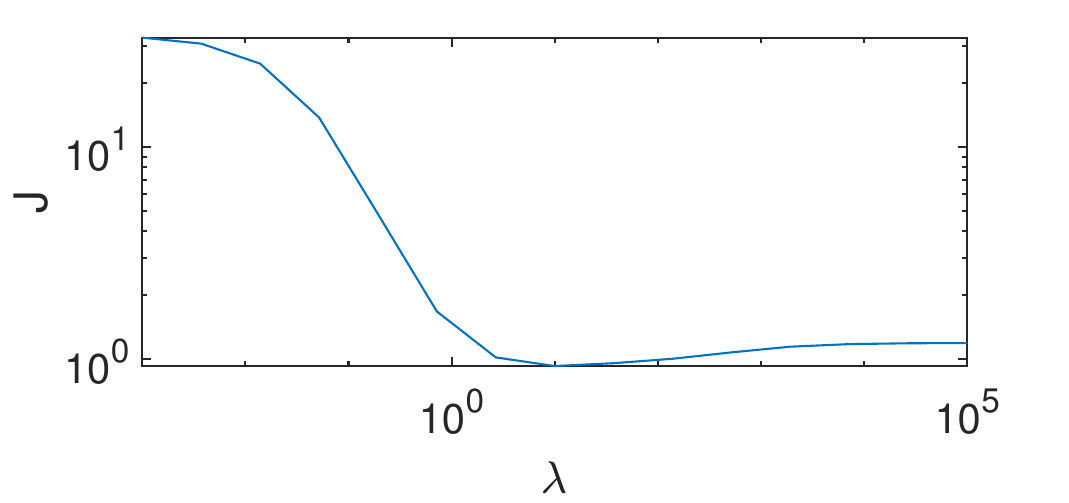}
    \caption{Performance of the regularized DDPC \eqref{equation: regularized DeePC} with IV \eqref{equation: IV closed loop} averaged over $200$ data sets for different $\lambda$ from $10^{-3}$ to $10^5$ with SNR = $25$dB.}
    \label{fig: lambda}
\end{figure}

The effect of $\lambda$ on the performance of RDDPC-IV is also studied, as shown in Fig. \ref{fig: lambda}. It can be observed that the control performance gets gradually improved as $\lambda$ decreases when $\lambda \ge 10$, which showcases the benefit of relaxation and offers more insight into results of Fig. \ref{fig:prediction performance}.

\section{Application To A Simulated Tubular Furnace}

As a key equipment in petro-chemical industry, the tubular furnace has been widely used for heating crude oil to a desired temperature before feeding into downstream units. As sketched in Fig. \ref{fig: furnace}, fuel gas is burnt to heat the crude oil in the tube, whereas an appropriate amount of air is needed to ensure the combustion efficiency. Thus, besides the outlet temperature of crude oil, the ${\rm O_2}$ content of stack gas shall be controlled for economic and efficient operations \cite{zeybek2006role}.
\begin{figure}[h]
    \centering
    \includegraphics[width=0.4\textwidth]{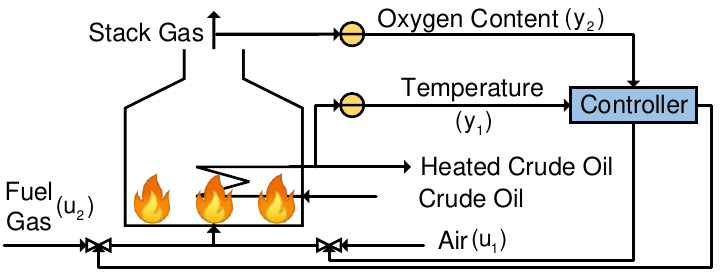}
    \caption{Structural diagram of industrial tubular furnace}
    \label{fig: furnace}
\end{figure}

As shown in Fig. \ref{fig: furnace}, the furnace system has two inputs, i.e. flow rates of natural gas ($u_1$, \SI{}{m^3/h}) and air ($u_2$, \SI{}{m^3/h}), and two outputs, i.e. outlet temperature of crude oil ($y_1$, \SI{}{\degreeCelsius}) and the oxygen content of stack gas ($y_2$, \SI{}{\%}). High-fidelity simulations of this two-input-two-output industrial furnace is enabled by the Fired Process Heater (FPH) simulator in the Honeywell UniSim Design Suite. Data collection and control are carried out with an interval time of $3$ mins due to the slow dynamics of the process. A coarsely tuned stabilizing controller $C(z)$ is parameterized as:
\begin{equation}
    \label{equation: controller matrices Unisim}
    \begin{aligned}
        &A_c=\begin{bmatrix}
            0&0&0&0\\0&1&0&0\\0&0&0&0\\0&0&0&1
        \end{bmatrix},~
        B_c=\begin{bmatrix}
            0&0.3260\\0&0.0802\\0.6250&0\\0.2990&0
        \end{bmatrix},\\
        &C_c=\begin{bmatrix}
            1&1&0&0\\0&0&1&1
        \end{bmatrix},~
        D_c=\begin{bmatrix}
            0&0\\0&0
        \end{bmatrix}.
    \end{aligned}
\end{equation}
A dataset $\{u^d(i),y^d(i)\}_{i=1}^N$ of size $N=5000$ is collected under closed-loop control, by varying the setpoint of $y_1$ between \SI{359}{\degreeCelsius} and \SI{367}{\degreeCelsius} and that of $y_2$ between $1.3\%$ and $2.7\%$. In online predictive control, the objective is to make $y_1$ track a square wave varying from \SI{362}{\degreeCelsius} to \SI{364}{\degreeCelsius}, stabilize $y_2$ around $2\%$ for high combustion efficiency. We set $L_p=L_f=70$, $Q={\rm diag}(1_{L_f}\otimes[10^{-4},5\times10^{-3}])$, and $R={\rm diag}(1_{L_f}\otimes[10^{-5},10^{-5}])$. The input constraint set $\mathbb{U}$ is set such that $u_1$ is maintained within [\SI{24000}{m^3/h}, \SI{34000}{m^3/h}], and $u_2 \in$ [\SI{2100}{m^3/h}, \SI{3100}{m^3/h}].

Using different control methods, outputs of the furnace under the same initial state are profiled in Fig. \ref{fig:Hysys}. Clearly, the proposed closed-loop DDPC-IV method achieves better tracking performance for $y_1$ than the generic SPC as well as the stabilizing controller $C(z)$, while maintaining $y_2$ near its desired value. This highlights the effectiveness of the proposed DDPC-IV scheme using closed-loop data.

\begin{figure}
    \centering
    \subfigure[Outlet temperature]{
        \includegraphics[width=0.45\textwidth]{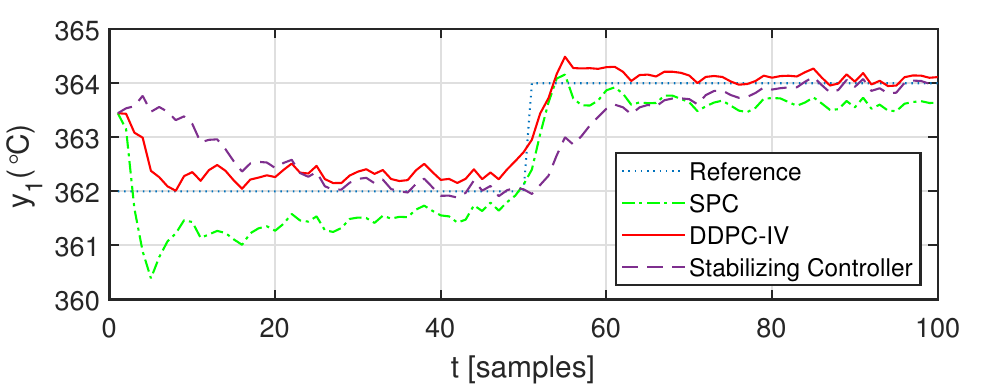}
    }
    \subfigure[${\rm O_2}$ content]{
        \includegraphics[width=0.45\textwidth]{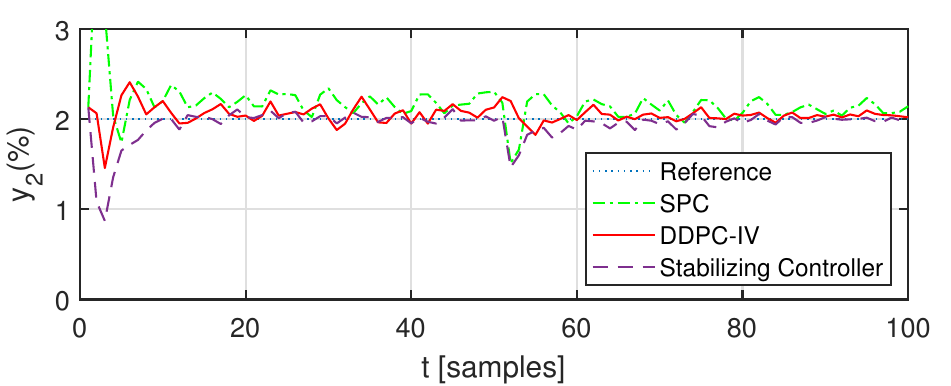}
    }
    \caption{Outputs of the tubular furnace system controlled by different controllers, including SPC, the proposed DDPC-IV and the stabilizing controller $C(z)$ designed as \eqref{equation: controller matrices Unisim}.}
    \label{fig:Hysys}
\end{figure}

\section{Conclusion}

In this paper, we proposed a new DDPC method with IV, which enables its implementation with closed-loop data. Two specific choices of IV based on closed-loop data were proposed inspired by works in closed-loop SID, which help to mitigate the noise effect and remain correlated with Hankel matrices of future inputs and outputs. Furthermore, an IV-inspired regularization scheme was proposed, where performance improvement can be achieved by balancing between control cost minimization and fitting a multi-step predictor from data. Numerical examples and application to a simulated furnace system were carried out to demonstrate the superior control performance of the proposed DDPC over the classical DDPC algorithms while using closed-loop data.

\addtolength{\textheight}{-12cm}   







\bibliographystyle{IEEEtran}
\bibliography{ref_bib}

\end{document}